\documentclass[12pt]{amsart}

\usepackage[margin=3cm]{geometry}

\usepackage{amssymb}

\usepackage{amsmath}

\usepackage{enumitem}

\usepackage{nccmath}

\usepackage{amsthm}
\theoremstyle{definition} \newtheorem{theorem}{Theorem}[section]
\theoremstyle{definition} 
\theoremstyle{definition} \newtheorem{lemma}[theorem]{Lemma}
\theoremstyle{definition} 
\theoremstyle{definition} \newtheorem{corollary}[theorem]{Corollary}
\theoremstyle{remark} 
\theoremstyle{remark} 
\theoremstyle{remark} 

\newcommand{\Zp}{\mathbb{Z}/p}

\begin{document}

\title[Bounding size of homotopy groups of spheres]{Bounding size of homotopy groups of spheres}

\author{Guy Boyde}
\address{Mathematical Sciences, University of Southampton, Southampton SO17 1BJ, United Kingdom}
\email{gb7g14@soton.ac.uk}

\subjclass[2010]{Primary 55Q40}
\keywords{homotopy groups of spheres, EHP Sequence}

\begin{abstract} Let $p$ be prime. We prove that, for $n$ odd, the $p$-torsion part of $\pi_q(S^{n})$ has cardinality at most $p^{2^{\frac{1}{p-1}(q-n+3-2p)}}$, and hence has rank at most $2^{\frac{1}{p-1}(q-n+3-2p)}$. For $p=2$ these results also hold for $n$ even. The best bounds proven in the existing literature are $p^{2^{q-n+1}}$ and $2^{q-n+1}$ respectively, both due to Hans-Werner Henn. The main point of our result is therefore that the bound grows more slowly for larger primes. As a corollary of work of Henn, we obtain a similar result for the homotopy groups of a broader class of spaces.
\end{abstract}

\maketitle

\section{Introduction}

Our goal is to bound the size of $\pi_q(S^n)$. Serre \cite{Serre2} showed that these groups are finitely generated abelian, and that they contain a single $\mathbb{Z}$-summand when $q=n$, or when $n$ is even and $q=2n-1$, and are finite otherwise. We may therefore restrict our attention to the $p$-torsion summand for each prime $p$. When $p$ is odd it suffices to consider $n$ odd, because, again by work of Serre \cite{Serre1}, $$\pi_q( S^{2n}_{(p)}) \cong \pi_{q-1}(S^{2n-1}_{(p)}) \oplus \pi_q( S^{4n-1}_{(p)}),$$ where $X_{(p)}$ denotes the localisation of the space $X$ at $p$.

For us, the \emph{rank} of a finitely generated module will be the size of a minimal generating set. Selick \cite{Selick} proved that the rank of $\pi_q(S^n_{(p)})$, regarded as a $\mathbb{Z}_{(p)}$-module, is at most~$3^{q^2}$. B\"odigheimer and Henn \cite{BodigheimerHenn} prove that the base-$p$ logarithm of the cardinality of the $p$-torsion part of $\pi_q(S^n)$ is at most $3^{(q-\frac{n}{2})}$, for $n$ odd and all primes $p$. By the same method, they prove that the rank of $\pi_q(S^n_{(p)})$ satisfies the same bound. In \cite{Henn}, Henn further improved this bound to $2^{q-n+1}$. In \cite{Iriye}, Iriye states the bound $3^\frac{q-n}{2p-3}$, which is similar to our Theorem \ref{main}, but gives no proof.

We will use the same machinery as all three of those papers, namely the EHP sequences of James \cite{James} and Toda \cite{Toda}. In particular, the new ideas in this paper are primarily combinatorial. A sub-exponential bound is not known, and \cite{BodigheimerHenn} note that to produce such a bound, one would have to introduce additional information from topology.

Our main result is as follows. Denote by $s_p(n,q)$ the base-$p$ logarithm of the cardinality of the $p$-torsion part of $\pi_q(S^n)$.

\begin{theorem} \label{main} For all natural numbers $q$, and $n$ odd, $$s_p(n,q) \leq 2^{\frac{1}{p-1}(q-n+3-2p)}.$$ If $p=2$ then this bound holds also for $n$ even.
\end{theorem}

As a corollary we obtain the weaker but simpler bound $s_p(n,q) \leq 2^{\frac{q-n}{p-1}}$. The $3-2p$ which appears in the original statement reflects the classical fact that the first $p$-torsion classes appear in the $(2p-3)$-rd stem. We can think of the bound as an exponential function of the stem $q-n$ in base $2^{\frac{1}{p-1}}$. The main advantage of this bound over its predecessors is that as $p$ becomes large, the base of the exponential approaches 1, so our bound grows more slowly for larger primes.

As in \cite{BodigheimerHenn}, we also obtain a bound on the rank, but we prefer to regard it as following from Theorem \ref{main}, by using that the rank of a finite $p$-torsion group is at most $\log_p$ of its cardinality.

\begin{corollary} \label{rankBound} For $n$ odd, the rank of the $p$-torsion part of $\pi_q(S^n)$ is at most $2^{\frac{q-n}{p-1}}$. If $p=2$, this bound holds also for $n$ even.
\end{corollary}

The bound proven by Henn in \cite{Henn} was a lemma used to establish results about the rank of the $p$-torsion part of $\pi_q(X)$ for $X$ any simply connected space of finite type. Our improvement to the bound feeds directly into the main theorem of that paper to give that a certain constant $c_p$ which appears there may be assumed to be at least $(\frac{1}{2})^{\frac{1}{p-1}}$ (Henn shows that it is at least $\frac{1}{2}$). This has the following corollary.

\begin{corollary} Let $X$ be a simply connected space of finite type, and suppose that the radius of convergence of the power series $\sum_{q=1}^\infty \dim_{\Zp}(H_q(\Omega X; \Zp))\cdot x^q$ is 1. Then $\sum_{q=1}^\infty \dim_{\Zp}(\pi_q(X; \Zp) \otimes \Zp )\cdot x^q$ has radius of convergence at least $(\frac{1}{2})^{\frac{1}{p-1}}$. In particular, the rank of the $p$-torsion part of $\pi_q(X)$ is at most $2^{\frac{q}{p-1}}$ for all but perhaps finitely many $q$.
\end{corollary}

The hypotheses of the above corollary are satisfied if, for example, the dimension of $H_q(\Omega X ; \Zp)$ is bounded above by a polynomial in $q$.

I would like to thank my supervisor, Stephen Theriault, for all of his help and support. I would also like to thank Hans-Werner Henn for his helpful correspondence - in particular for drawing my attention to the methods of his paper \cite{Henn}, and for making me aware of the paper \cite{FlajoletProdinger} of Flajolet and Prodinger, which lead indirectly to the idea for this paper.

\section{Approach} \label{approach}

It will be convenient for us to think of a stem in the homotopy groups of spheres as a combinatorial object, so we say that the $k$-th \emph{stem} is the set $\{(n,q) \in \mathbb{N} \times \mathbb{N} \ \lvert \ q-n=k \}$. The \emph{negative} stems, for example, are those for which $k<0$.

Fix a prime $p$, and assume henceforth that all spaces are localized at $p$. For a space $X$, let $J_k(X)$ denote the $k$-th stage of the James Construction on $X$. For $n$ odd, we have the following ($p$-local) fibrations, from \cite{Toda} and \cite{James}: $$J_{p-1}(S^{n-1}) \longrightarrow \Omega S^{n} \longrightarrow \Omega S^{p(n-1)+1}, \textrm{ and}$$ $$S^{n-2} \longrightarrow \Omega J_{p-1}(S^{n-1}) \longrightarrow \Omega S^{p(n-1)-1}.$$

The long exact sequences on homotopy groups induced  by these fibrations are called \emph{EHP-sequences}. The following inequalities are proven in the first lemma of \cite{BodigheimerHenn}. The first inequality is obtained by considering the EHP-sequences, assuming that all groups are finite, and the second inequality accounts for the possibility that one of the groups is not finite, using knowledge of the relative homotopy groups $\pi_q(\Omega^2 S^{n}, S^{n-2})$, as in for example Appendix 2 of \cite{Husemoller}. When $p=2$ the situation is simpler: the above fibrations are just the odd and even cases of $$S^{n-1} \longrightarrow \Omega S^{n} \longrightarrow \Omega S^{2n-1}.$$

\begin{lemma} \label{inequalities} For all $q$, and $n$ odd:

\begin{enumerate}
    \item $s_p(n,q) \leq s_p(p(n-1)+1,q)+s_p(p(n-1)-1,q-1)+s_p(n-2,q-2)$ if $q \neq p(n-1)$.
    \item $s_p(n,q) \leq 1+s_p(n-2,q-2)$ if $q = p(n-1)$.
\end{enumerate}

When $p=2$, we no longer need to restrict to odd $n$, and the first inequality can be replaced by $$s_2(n,q) \leq s_2(2n-1,q)+s_2(n-1,q-1).$$
\end{lemma}

These inequalities will be used to prove Theorem \ref{main}. It is worth noting that the extent to which the inequalities fail to be equalities is measured by the size of the images of the boundary maps (equivalently, the kernels of the suspensions) in the EHP sequences. In some sense, therefore, the extent to which our bound fails to be sharp is measuring the aggregate size of the images of EHP boundary maps.

\section{Limitations of our approach}

If in Lemma \ref{inequalities}, one replaces the inequalities with equalities, and regards this as an inductive definition of integers $t_p(n,q)$, then necessarily $t_p(n,q)$ is the best upper bound that can be obtained for $s_p(n,q)$ using that lemma. In \cite{Henn}, Henn defines inductively integers $b_2(n,k)$. He shows that $t_2(n,q)=b_2(q-2,n)$ (note that $n$ has switched roles). In \cite{FlajoletProdinger}, Flajolet and Prodinger study a combinatorially defined sequence $H_n$. By definition, $t_2(2,q)=b_2(q-2,2)=H_{q-2}$, a fact I was made aware of by Henn. Flajolet and Prodinger obtain an asymptotic estimate $H_q \sim K \cdot \nu^q$, giving formulas for $K$ and $\nu$ and computing both to 15 decimal places. To 3 decimal places, $K$ is 0.255, and $\nu$ is 1.794. They remark that $H_q$ (which is equal to $t_2(2,q+2)$) is at least $F_q$, where $F_q$ is the $q$-th Fibonacci number. We will not do so here, but one can show by induction that $t_p(3,2p+j(4p+5)) \geq F_{2j+1}$ for $j \geq 0$. Inductively applying $t_p(n-2,q-2) \leq t_p(n,q)$ then gives that $t_p(n,2p+j(4p+5)+n-3) \geq F_{2j+1}$ for all odd $n \geq 3$. We therefore have the following.

\begin{corollary} Any bound on the base-$p$ logarithm of the cardinality of the $p$-primary part of $\pi_{2p+j(4p+5)+n-3}(S^n)$ (or the rank of that group) which can be obtained from Lemma \ref{inequalities} is greater than or equal to $F_{2j+1}$. In particular, any base for an exponential bound which can be obtained from Lemma \ref{inequalities} is at least $\phi^{\frac{2}{4p+5}}$, where $\phi$ denotes the golden ratio. 
\end{corollary}

\section{Proof of Theorem \ref{main}}

\begin{proof}[Proof of Theorem \ref{main}] We will actually prove the slightly stronger result that $s_p(n,q) \leq 2^{\lfloor \frac{1}{p-1}(q-n+3-2p) \rfloor}$. The floor function forces the exponent to be an integer, which will be useful in the proof. We will use a (slightly modified) strong double induction over stems. More precisely, the proof of the result for $(n,q)$ will use the result for all $(m,r)$ with $r-m<q-n$ (that is, on lower stems) and for $(1,q-n+1)$ (that is, the entry at the base of the stem on which $(n,q)$ lies). A proof using the other lower entries on the same stem in the induction is possible, but results in a more unwieldy inductive hypothesis. The case $p=2$, $n$ even will be treated at the end.

Suppose first that $q \leq n$. In this case, $\pi_q(S^n)$ is torsion free (indeed, it is zero for $q<n$) so $s_p(n,q) = 0$. This proves the result for all non-positive stems. The higher homotopy groups of $S^1$ are trivial (since it has a contractible universal cover) so $s_p(1,q)=0$ for all $q$. This proves the base case of each stem.

It remains only to treat an inductive step on a positive stem. Thus, let $(n,q) \in \mathbb{N} \times \mathbb{N}$ with $n$ odd, and suppose that the result is proven for all $(m,r)$ with $r-m<q-n$. Consider the two inequalities of Lemma \ref{inequalities}. We wish to apply the first inequality inductively down the stem to bound $s_p(n,q)$ by a sum of terms on lower stems and $s_p(1,q-n+1)$, which is zero by the discussion above. The only complicating factor is the second inequality, which may require us to add one to our bound at certain steps. However, the second case of Lemma \ref{inequalities} can occur at most once per stem, so at worst we will have to add one to the bound that we would obtain if the first case of the Lemma held everywhere. More precisely, we obtain $$s_p(n,q) \leq 1 + \sum_{i=0}^{\frac{1}{2}(n-3)} (s_p(p(n-2i-1)+1,q-2i)+s_p(p(n-2i-1)-1,q-2i-1)).$$

Each $s_p(m,r)$ is an integer. Therefore, for those $(m,r)$ for which we inductively have $s_p(m,r) \leq 2^{\lfloor \frac{1}{p-1}(r-m+3-2p) \rfloor}$, we actually have the sightly stronger statement that $s_p(m,r) \leq \lfloor 2^{\lfloor \frac{1}{p-1}(r-m+3-2p) \rfloor} \rfloor$. Including this fact into the above inequality, we find that 

\begin{equation}\label{star} \medmath{s_p(n,q) \leq 1 + \sum_{i=0}^{\frac{1}{2}(n-3)} (\lfloor 2^{\lfloor \frac{1}{p-1}(q-2i-(p(n-2i-1)+1)+3-2p) \rfloor} \rfloor+\lfloor 2^{\lfloor \frac{1}{p-1}(q-2i-1-(p(n-2i-1)-1)+3-2p) \rfloor} \rfloor). \tag{$\ast$}} \end{equation}

Notice that the value of the floor function on an integer power of 2 is given by $$\lfloor 2^i \rfloor = \begin{cases} 2^i & i \geq 0 \\
0 & i < 0.
\end{cases}$$

We now bound this summation by another where the nonzero exponents are distinct integers. More precisely, adding $1-\frac{1}{p-1}$ to the exponent of the second term in \eqref{star} (inside the floor function) gives that

\begin{equation*}\medmath{
\begin{split}
s_p(n,q) & \leq 1 + \sum_{i=0}^{\frac{1}{2}(n-3)} (\lfloor 2^{\lfloor \frac{1}{p-1}(q-2i-(p(n-2i-1)+1)+3-2p) \rfloor} \rfloor+\lfloor 2^{\lfloor \frac{1}{p-1}(q-(2i+1)-(p(n-(2i+1)-1)+1)+3-2p) \rfloor} \rfloor) \\
& = 1 + \sum_{i=0}^{n-2} (\lfloor 2^{\lfloor \frac{1}{p-1}(q-i-(p(n-i-1)+1)+3-2p) \rfloor} \rfloor. \end{split}}
\end{equation*}

In particular, $s_p(n,q)$ is at most one greater than a sum of powers of 2. It suffices to show that those powers of 2 that are not killed off by the outer floor function are all distinct and strictly smaller than $2^{\lfloor \frac{1}{p-1}(q-n+3-2p) \rfloor}$, because $\sum_{i=0}^{k-1}2^i = 2^{k}-1$. 

To see that they are distinct, notice that changing $i$ by 1 changes the exponent by 1. To see that they are strictly smaller than $2^{\lfloor \frac{1}{p-1}(q-n+3-2p) \rfloor}$, consider the largest power occurring in the summation, which is the $i=n-2$ term. Its exponent rearranges to $\lfloor \frac{1}{p-1}(q-n+3-2p) - 1 \rfloor = \lfloor \frac{1}{p-1}(q-n+3-2p) \rfloor - 1$, as required. This completes the proof for $n$ odd.

It remains to treat the case $p=2$, $n$ even. Since the simplification at $p=2$ in Lemma \ref{inequalities} holds for all $n$, the above proof may be repeated without restricting to $n$ odd, and doing so gives the result for all $n$.
\end{proof}

\bibliographystyle{amsalpha}

\end{document}